\documentclass{amsart}
\usepackage[colorlinks]{hyperref}
\renewcommand\eqref[1]{(\ref{#1})} 

 \newtheorem{thm}{THEOREM}[section]
 \newtheorem{cor}[thm]{Corollary}
 
 \newtheorem{prop}[thm]{Proposition}
 \newtheorem{claim}[thm]{Claim}
 \newtheorem{rem}[thm]{Remark}
 \newtheorem{example}[thm]{Example}
 
 \newcommand{\set}[1]{\left\{#1\right\}}

\newcommand \be     {\begin{equation}}
\newcommand \ee     {\end{equation}}
\newcommand {\RR} {\mathbb{R}}
\newcommand {\CC} {\mathbb{C}}
\newcommand\R{{\mathbb R}}
\newcommand \Rn    {\mathbb{R}^n}
\newcommand \eps   {\epsilon}

\newcommand \Ncal   {\mathcal N}
\newcommand \Hcal   {\mathcal H}
\newcommand \Dcal   {\mathcal D}

 \def\br#1{{\left[{#1}\right]}}
 \def\n#1{{\left\|{#1}\right\|}}
 \def\jp#1{{\left\langle{#1}\right\rangle}}
 \def\p#1{{\left({#1}\right)}}
 \def\abs#1{{\left|{#1}\right|}}
\def\Im{\operatorname{Im}}

\def\supp{\operatorname{supp}}

\newcommand{\MX}{\mathcal{X}}

  \newtheorem*{newprop2.2}{PROPOSITION 2.2}
  \newtheorem*{newprop2.3}{PROPOSITION 2.3}
  \newtheorem{assume}[thm]{ASSUMPTION}
 \numberwithin{equation}{section}
\begin{document}

\title [SPECTRAL IDENTITIES AND SMOOTHING ESTIMATES ]
{SPECTRAL IDENTITIES AND SMOOTHING ESTIMATES FOR EVOLUTION OPERATORS}
\author{ Matania Ben-Artzi }
\address{Institute of Mathematics, Hebrew University, Jerusalem
91904, Israel}
\email{mbartzi@math.huji.ac.il}
\author{ Michael Ruzhansky }
\address{Department of Mathematics, Imperial College, London SW72AZ,UK}
\email{m.ruzhansky@imperial.ac.uk}
\author{ Mitsuru Sugimoto }
\address{Graduate School of Mathematics, Nagoya University, Nagoya 464-8602, Japan}
\email{sugimoto@math.nagoya-u.ac.jp}
\thanks{The second
 author was supported in parts by the  EPSRC grant EP/R003025/1 and by the Leverhulme
Grant RPG-2017-151.
 The third author is partially supported by Grant-in-aid for Scientific Research from JSPS (No.~26287022 and No.~26610021).}


\date{\today}
\begin{abstract} Smoothing (and decay) spacetime estimates are discussed for evolution groups of self-adjoint operators in an abstract setting. The basic assumption is the existence (and weak continuity) of the spectral density in a functional setting.
 Spectral identities for the time evolution of such operators
are derived, enabling results concerning ``best constants'' for smoothing estimates.  When combined with suitable
``comparison principles'' (analogous to those established in ~\cite{RS2012}), they yield
 smoothing estimates for classes of functions of the
operators .

A important particular case is the derivation of global spacetime estimates for a perturbed operator $H+V$ on the basis of its comparison with the unperturbed operator $H.$

A number of applications are given, including smoothing estimates for fractional Laplacians, Stark Hamiltonians and
Schr\"odinger operators with potentials.

\end{abstract}

\maketitle

\tableofcontents\newpage
\section{Introduction}
In this paper we present an abstract framework for global spacetime and smoothing estimates for evolution groups generated by self-adjoint operators. In particular, this approach leads to such estimates for various classes of pseudodifferential operators.

     Global spacetime estimates for solutions of partial differential evolution equations (such as the wave equation or the Schr\"{o}dinger equation) have already become a fundamental tool in the investigation of such equations, subject to linear and nonlinear perturbations. It suffices to mention in this context the Strichartz estimates ~\cite{strichartz,tao} or weighted-$L^2$ estimates ~\cite{ba9}. These global estimates present a twofold aspect,   expressing both an additional regularity of the solutions  and  their decay properties at large time. Additionally, the local ``smoothing effects'' (such as the ``$\frac12-$derivative gain'' in the case of the Schr\"{o}dinger equation) have been incorporated into the estimates ~\cite{ConstSaut,kenig,vega}.

 The class of equations for which such estimates have been established has been substantially extended  in the last twenty years, as can be seen in ~\cite{ba6,ba7,ba8,ba10,ba5,ba11,ba12,ancona,doi,hoshiro1,hoshiro2,ky} and references therein.

     The equivalence of many different estimates has been observed in ~\cite{RS2012}, by an application of canonical transforms and comparison principles. 
     At the same time, it can be also used to obtain critical versions of the limiting absorption principle and smoothing estimates \cite{RS2012cmp}.
     
     In this paper we advance further this approach, in a unifying abstract setting. In particular, various functions (in the functional-analytic sense) of an operator $H$ admit such estimates, based solely on their validity for $H.$ The treatment here enables sharp spacetime and smoothing estimates for evolution groups of pseudodifferential operators, with initial data that may be localized in ``energy space''. We do not need to exclude the possibility of singular spectrum embedded in the absolutely continuous part of the spectrum.

     As in the case of many useful differential inequalities (such as Hardy's inequality), it is of interest to establish the ``best (or optimal) constant'' . This has been done in the case of the free Schr\"{o}dinger operator in ~\cite{simon} and further extended in ~\cite{BezSS1,BezSS2}.  In the abstract framework used here, we are able to provide such a best constant in a wide array of spacetime estimates.

     The paper is organized as follows.

     In Section ~\ref{secbasicsetup} the basic abstract setup and notation are introduced.

     Section ~\ref{secglobalevolution} presents the main global estimates in the abstract formulation, as well as the ``best constant'' result
     of Theorem ~\ref{thmbestconst}.

     Section ~\ref{sec-comparison} deals with comparison principles. Generally speaking, if the evolution of a self-adjoint operator $H$ can be globally (namely, in space-time) estimated, then the same is true for another self-adjoint operator $\widetilde H,$ provided the spectral derivatives of the two operators can be ``compared'' (Theorem ~\ref{THM:sp-comparison-gen-op}). As a special case, we obtain a ``comparison'' between a self-adjoint operator and its potential perturbation (Theorem ~\ref{THM:sp-comparison-gen-op-pert}).

     Finally, in Section ~\ref{sec-examples} we give a few examples, such as fractional Laplacians and Schr\"{o}dinger operators. We did not try to present the ``most complicated'' cases, but rather convey the general ideas as applied to well-known operators of mathematical physics.

     \section{The basic setup and notation}\label{secbasicsetup}

     \textbf{Notation.} The following notations are used throughout the paper.
     \begin{itemize}
     \item   $\jp{x}=(1+|x|^2)^\frac12.$
     \item The Fourier transform in $\Rn:$
        $$\mathcal{F}f(\xi)=\widehat{f}(\xi)=(2\pi)^{-\frac{n}{2}}\int_{\Rn}f(x)e^{-i\xi x}\,dx.$$
        \item It is useful in multi-variable formulas to indicate by an index  a variable in integration:

        $$L^2(J_\lambda)=\set{v(\lambda)\,;\,\,\int_J|v(\lambda)|^2\,d\lambda<\infty}.$$
        \item Let $I\subseteq\RR$ be an interval and let $Y$ be a Hilbert space. Then
          $$L^2(I,Y)=\set{g:I\to Y\,;\,\,  \|g\|^2_{L^2(I,Y)}=\int_I\|g(t)\|^2_Y\,dt<\infty}.$$
        \item The space of bounded linear operators on a space $X$ into a space $Y$ is designated as $B(X,Y)$ with (uniform) operator norm $\|\cdot\|_{B(X,Y)}.$
     \end{itemize}

    Let $H$ be a self-adjoint operator in a Hilbert space $\Hcal.$ The scalar product and norm in $\Hcal$ are denoted respectively by $(\cdot,\cdot)_{\Hcal}$ and $\|\cdot\|_{\Hcal}.$

    We denote by $\set{E(I),\,\,I\subseteq\RR \,\,\mbox{an open interval}}$ the associated  spectral family of $H.$  We use the common simplified notation
     $E(\lambda)=E(-\infty,\lambda).$

    By $P_{ac}(H)$ we denote the projection on the absolutely continuous subspace.

    Let $\Sigma_{ac}\subseteq\RR$ be the absolutely continuous spectrum of $H.$ Note that we are not assuming the absence of singular spectrum embedded in the absolutely continuous spectrum. Thus, we allow $P_{ac}(H)\neq E(\Sigma_{ac}).$

    Let $\MX\subseteq\Hcal$ be an embedded dense subspace with a stronger norm $\|\cdot\|_{\MX}.$ If $\MX^\ast$ is the dual space, we have the canonical inclusion $\MX\subseteq\Hcal\subseteq\MX^\ast.$ We denote by $\jp{\cdot,\cdot}$  the $\MX^\ast,\MX$ pairing.

   Let $J\subseteq\RR$ be an open set.  The \textbf{basic tool in our treatment}  is the hypothesis that the weak spectral derivative
    $$A(\lambda)=\frac{d}{d\lambda}\Big(E(\lambda)P_{ac}(H)\Big),\quad \lambda\in J,$$
   exists and
        is bounded from  $\MX$ into $\MX^\ast.$ Thus
       \be\label{eqAXX*}\jp{A(\lambda)f,g}=\frac{d}{d\lambda}\Big(E(\lambda)P_{ac}(H)f,g\Big)_{\Hcal},\quad f,g\in \MX,\,\,\lambda\in J.\ee
            Since $<A(\lambda)f,g>$ is a nonnegative bilinear form we have
            $$|\jp{A(\lambda)f,g}|\leq \,\, \jp{A(\lambda)f,f}^{\frac12}\cdot \jp{A(\lambda)g,g}^{\frac12},$$
            hence
            \be\label{eqnormAXXast}\|A(\lambda)\|_{B(\MX,\MX^\ast)}=\sup\limits_{\|f\|_{\MX}=1}\jp{A(\lambda)f,f}^{\frac12}.\ee
       \begin{rem}\label{remEalambda} Note that for every $a>-\infty$ we clearly have
       \be\label{eqAXX*a}\jp{A(\lambda)f,g}=\frac{d}{d\lambda}\Big(E((a,\lambda))P_{ac}(H)f,g\Big)_{\Hcal},\quad f,g\in \MX,\,\,a<\lambda\in J.\ee
             \end{rem}
             \begin{rem}\label{rem-XXast} When dealing with perturbations (Subsection ~\ref{subsec-pert}) we shall need to use a more general setting.
             \end{rem}

       We shall find it simpler (and clearer) to formulate our results in the abstract setting. A typical concrete example is $\Hcal=L^2(\Rn)$ and $\MX=L^2_s(\Rn_x),\,s>0,$ the weighted $L^2$ space, which is the usual choice in many applications. This space is defined by:
       \be L^2_s(\Rn_x)=\set{f(x)\,;\,\,\|f\|_{0,s}^2=\int\limits_{\Rn}\jp{x}^{2s}|f(x)|^2\,dx<\infty,\,\,s\in\RR}.
       \ee

       In this case we denote the $\MX^\ast,\MX$ pairing simply as
       $(\cdot,\cdot)_{L^2_{-s},L^2_s},$ and the corresponding norm of the spectral derivative by $\|A(\lambda)\|_{s,-s}.$

\section{Global smoothing estimates}\label{secglobalevolution}
The following assumption is fundamental in what follows.

 \begin{assume}\label{assumption} Let $J\subseteq\RR$ be an open set .  The operator-valued function $A(\lambda):\MX\hookrightarrow \MX^\ast,$ is weakly continuous on $J.$

 In particular, it is also locally bounded on $J$ (in the uniform operator topology).
 \end{assume}

The class of  operators studied here involves
  two real-valued functions:
  \begin{assume}\label{assumptionasigma}
 $$\sigma(\lambda)\,\,\,\, \mbox{is continuous on}\,\,J.$$
 $$\aligned\mbox{ There exists a finite set}\,\, \Ncal\subseteq J \,\,\mbox{so that}\,\, a(\lambda)\in C^1(J\setminus\Ncal),\,\, \,\,\mbox{and}\\
 a'(\lambda)> 0,\,\,\lambda\in J\setminus\Ncal.\hspace{100pt}
 \endaligned $$
\end{assume}
 We first establish  useful identities for the unitary group $e^{ita(H)}.$ They are essentially an expression of the duality of ``time'' and the ``spectral parameter'' (or ``energy'').
 \begin{prop}\label{propidentity} Assume the conditions of Assumptions ~\ref{assumption}-~\ref{assumptionasigma}.

Let $\phi,\psi\in\MX.$ Then
 \begin{enumerate}
 \item \be\label{eq1identity}\aligned \|(\sigma(H)e^{ita(H)}P_{ac}(H)E(J)\phi,\psi)_{\Hcal}\|_{L^2(\RR_t)}\hspace{100pt} \\=
     \sqrt{2\pi}\,\Big\|\frac{\sigma(\lambda)}{a'(\lambda)^\frac12}
 \jp{A(\lambda)\phi,\psi}\Big\|_{L^2(J_\lambda)},
 \endaligned\ee
 and
 \item
  \begin{equation}\label{eq2identity}
 \aligned \Big\|\sigma(H)e^{ita(H)}P_{ac}(H)E(J)\phi\Big\|_{ L^2(\RR_t,\MX^\ast)}\hspace{100pt}\\=
 \sqrt{2\pi}\,\Big\|\frac{\sigma(\lambda)}{a'(\lambda)^\frac12}
 (A(\lambda)\phi)\Big\|_{ L^2(J_\lambda,\MX^\ast)}.
 \endaligned\end{equation}
 \end{enumerate}
 \end{prop}
 \begin{proof} For simplicity we take $\Ncal=\emptyset$ in   Assumption ~\ref{assumptionasigma}. Otherwise, the set $J$ should be replaced by $J\setminus\Ncal.$

  To prove ~\eqref{eq1identity} we use the definition of $A(\lambda)$ and the spectral calculus to get
 \be
 (\sigma(H)e^{ita(H)}P_{ac}(H)E(J)\phi,\psi)_{\Hcal}=\int\limits_{J}e^{ita(\lambda)}\sigma(\lambda)\jp{A(\lambda)\phi,\psi}d\lambda.
 \ee
 Defining a new variable $\eta=a(\lambda)$ yields, with $\widetilde{J}=\set{\eta=a(\lambda)\,;\,\lambda\in J},$ 
 $$
 (\sigma(H)e^{ita(H)}P_{ac}(H)E(J)\phi,\psi)_{\Hcal}=\int\limits_{\widetilde{J}}e^{it\eta}\sigma(a^{-1}(\eta))
 \jp{A(a^{-1}(\eta))\phi,\psi}\frac{d\eta}{a'(a^{-1}(\eta))}.
 $$
  Invoking Plancherel's theorem to this equality yields
 \be\label{eqtimelambda}\aligned
 \|(\sigma(H)e^{ita(H)}P_{ac}(H)E(J)\phi,\psi)_{\Hcal}\|_{L^2(\RR_t)}^2\hspace{150pt}\\ \\=
 2\pi\int\limits_{\widetilde{J}}\frac{|\sigma(a^{-1}(\eta))\jp{A(a^{-1}(\eta))\phi,\psi}|^2}{a'(a^{-1}(\eta))}
 \frac{d\eta}{a'(a^{-1}(\eta))}=2\pi\int\limits_{J}\frac{\sigma(\lambda)^2}{a'(\lambda)}
| \jp{A(\lambda)\phi,\psi}|^2\,d\lambda.
 \endaligned\ee
 This concludes the proof of ~\eqref{eq1identity}.

 To prove ~\eqref{eq2identity} we apply again the spectral decomposition of $H$ and the definition of $A(\lambda),$
 \be\label{eq2identitya}\aligned
 \Big\|\sigma(H)e^{ita(H)}P_{ac}(H)E(J)\phi\Big\|_{ L^2(\RR_t,\MX^\ast)}\hspace{100pt}\\
 =\sup\limits_{\|g\|_{ L^2(\RR_t,\MX)}=1}\Big|\int_\RR\Big(\sigma(H)e^{ita(H)}P_{ac}(H)E(J)\phi,
 g(t)\Big)_{\Hcal}dt\Big|\\
 =\sup\limits_{\|g\|_{ L^2(\RR_t,\MX)}=1}\Big|\int_\RR\int_{J_\lambda}\sigma(\lambda)
 e^{ita(\lambda)}\jp{A(\lambda)\phi,
 g(t)}\,d\lambda dt\Big|\\
 =\sup\limits_{\|g\|_{ L^2(\RR_t,\MX)}=1} \sqrt{2\pi}\,\Big|\int_{J_\lambda}\sigma(\lambda)
 \jp{A(\lambda)\phi,
 \widetilde{g}(a(\lambda))}\,d\lambda \Big|,
 \endaligned\ee
  where the one-dimensional (with respect to $t$) Fourier transform is
  \be\label{eqwideg}\widetilde{g}(\theta)=\frac{1}{\sqrt{2\pi}}\int_\RR e^{-it\theta}g(t)\,dt.\ee
  As above we introduce the new variable $\eta=a(\lambda)$ and use the Plancherel theorem to obtain,
  \be\label{eq2identityb1}\aligned
 \Big\|\sigma(H)e^{ita(H)}P_{ac}(H)E(J)\phi\Big\|_{ L^2(\RR_t,\MX^\ast)}\hspace{100pt}\\
 =\sup\limits_{\|\widetilde{g}\|_{ L^2(\RR_\eta,\MX)}=1} \sqrt{2\pi}\,\Big|\int_{\widetilde{J}}\frac{\sigma(a^{-1}(\eta))}{a'(a^{-1}(\eta))}
 \jp{A(a^{-1}(\eta))\phi,
 \widetilde{g}(\eta)}\,d\eta,
 \endaligned\ee
 hence
  $$
  \Big\|\sigma(H)e^{ita(H)}P_{ac}(H)E(J)\phi\Big\|_{\MX^\ast\bigotimes L^2(\RR_t)}^2=2\pi\Big|\int_{\widetilde{J}}\Big(\frac{\sigma(a^{-1}(\eta))}{a'(a^{-1}(\eta))}\Big)^2
 \|A(a^{-1}(\eta))\phi\|_{\MX^\ast}^2
 \,d\eta \Big|,
  $$
    from which
  ~\eqref{eq2identity} readily follows.
 \end{proof}
 The method of proof used for Proposition ~\ref{propidentity} in conjunction with the role of $A(\lambda)$ in the spectral calculus enable us to obtain a general spacetime estimate for any initial data $\phi\in\Hcal.$
 \begin{thm}\label{thmbasicspacetime}
 Assume the  conditions of Assumptions ~\ref{assumption}-~\ref{assumptionasigma} and let $\phi\in\Hcal.$ Then
 \begin{multline}\label{eqbasicspacetime}
  \Big\|\sigma(H)e^{ita(H)}P_{ac}(H)E(J)\phi\Big\|_{ L^2(\RR_t,\MX^\ast)}
\hspace{200pt}\\ \leq
 \sqrt{2\pi}\sup\limits_{\lambda\in {J\setminus\Ncal}}\Big\{\frac{|\sigma(\lambda)|}{a'(\lambda)^\frac12}
 \|A(\lambda)\|^\frac12_{B(\MX,\MX^\ast)}\Big\}\|P_{ac}(H)E(J)\phi\|_{\Hcal}.
\end{multline} 
 \end{thm}
 \begin{proof} For simplicity of the presentation we assume as above that $\Ncal=\emptyset.$ As in ~\eqref{eq2identitya} we write, assuming first that $\phi\in\MX,$
 \be\label{eq2identityb}\aligned
 \Big\|\sigma(H)e^{ita(H)}P_{ac}(H)E(J)\phi\Big\|_{ L^2(\RR_t,\MX^\ast)}\hspace{100pt}\\
 =\sup\limits_{\|g\|_{ L^2(\RR_t,\MX)}=1} \sqrt{2\pi}\,\Big|\int_{J}\sigma(\lambda)
 \jp{A(\lambda)\phi,
 \widetilde{g}(a(\lambda))}\,d\lambda \Big|.
 \endaligned\ee
 Since $A(\lambda)$ is a nonnegative form we have
   $$\Big|\jp{A(\lambda)\phi,
 \widetilde{g}(\cdot,a(\lambda))}\Big|^2\leq \,\,\jp{A(\lambda)\phi,\phi}\cdot \jp{A(\lambda)\widetilde{g}(a(\lambda)),\widetilde{g}(a(\lambda))}.$$
 Inserting this in ~\eqref{eq2identityb} and using the Cauchy-Schwarz inequality yields
 \be\label{eq2identityc}\aligned
 \Big\|\sigma(H)e^{ita(H)}P_{ac}(H)E(J)\phi\Big\|^2_{ L^2(\RR_t,\MX^\ast)}\hspace{100pt}\\
 \leq 2\pi\sup\limits_{\|g\|_{ L^2(\RR_t,\MX)}=1}\int_{J}\sigma(\lambda)^2\jp{A(\lambda)\widetilde{g}(a(\lambda)),\widetilde{g}(a(\lambda))}\,d\lambda
 \cdot\int_{J}\jp{A(\lambda)\phi,\phi}\,d\lambda\\=
 2\pi\sup\limits_{\|g\|_{ L^2(\RR_t,\MX)}=1}\int_{J}\sigma(\lambda)^2\jp{A(\lambda)\widetilde{g}(a(\lambda)),\widetilde{g}(a(\lambda))}\,d\lambda
 \cdot
 \|P_{ac}(H)E(J)\phi\|^2_{\Hcal},
 \endaligned\ee
 where the spectral theorem was used in the last equality.

  Employing the same change of variable $\eta=a(\lambda)$ as above we get
  \be\label{eq2identityd}\aligned
  \int_{J}\sigma(\lambda)^2\jp{A(\lambda)\widetilde{g}(a(\lambda)),\widetilde{g}(a(\lambda))}\,d\lambda\\
  \leq\Big|\int_{\widetilde{J}}\frac{\sigma(a^{-1}(\eta))^2}{a'(a^{-1}(\eta))}
 \|A(a^{-1}(\eta))\|_{B(\MX,\MX^\ast)}\|\widetilde{g}(\eta)\|_{\MX}^2
 \,d\eta \Big|\\
 \leq \sup\limits_{\lambda\in J}\Big\{\frac{\sigma(\lambda)^2}{a'(\lambda)}
 \|A(\lambda)\|_{B(\MX,\MX^\ast)}\Big\} \int_{\widetilde{J}}\|\widetilde{g}(\cdot,\eta)\|_{\MX}^2\,d\eta.
  \endaligned\ee
  Observe that by the Plancherel theorem
    $$\int_{\widetilde{J}}\|\widetilde{g}(\eta)\|_{\MX}^2\,d\eta\leq \int_{\RR}\|\widetilde{g}(\eta)\|_{\MX}^2d\eta=\|g\|^2_{ L^2(\RR_t,\MX)}=1.$$
    Plugging ~\eqref{eq2identityd} into ~\eqref{eq2identityc} we obtain ~\eqref{eqbasicspacetime}, still under the assumption that $\phi\in\MX.$ However, this assumption can be dropped due to the density of $\MX$ in $\Hcal.$
 \end{proof}
 \begin{rem} Note that in the above statements we could ``absorb'' $E(J)$ into $\sigma(H),$ but we preferred to emphasize the localization aspect of the estimates (in ``energy'' space) and to leave $\sigma(\lambda)$ as a continuous function.
         \end{rem}

   It will be expedient to give explicit statements  for ~\eqref{eq2identity} and ~\eqref{eqbasicspacetime} in the case of the weighted space, $\MX=L^2_s(\Rn_x).$ This is done in the following corollary.

\begin{cor} Consider the case where
 $\Hcal=L^2(\Rn)$ and $\MX=L^2_s(\Rn_x),\,s>0,$ the weighted $L^2$ space.
    Let $J\subseteq\RR$ be an open set and let $\phi\in L^2_s(\Rn).$  Then, under  Assumptions ~\ref{assumption} and ~\ref{assumptionasigma}, the following holds:
     \begin{itemize}
     \item We have
     \begin{equation}
\label{basic} \aligned \Big\|\jp{x}^{-s}\sigma(H)e^{ita(H)}P_{ac}(H)E(J)\phi\Big\|_{L^2(\RR_t\times\Rn_x)}\hspace{100pt}\\=
 \sqrt{2\pi}\,\Big\|\frac{\sigma(\lambda)}{|a'(\lambda)|^\frac12}
 \jp{x}^{-s}(A(\lambda)\phi)(x)\Big\|_{L^2(J_\lambda\times\Rn_x)}.
 \endaligned\end{equation}
 \item
     Let
     \be
      \|A(\lambda)\|_{s,-s}:=\sup\limits_{\|\psi\|_{0,s}=1}(A(\lambda)\psi,\psi)_{L^2_{-s},L^2_s}.
     \ee
     Then for all $\phi\in L^2(\Rn_x),$
     \be\label{estphilambda}
     \aligned \Big\|\jp{x}^{-s}\sigma(H)e^{ita(H)}P_{ac}(H)E(J)\phi\Big\|_{L^2(\RR_t\times\Rn_x)}\hspace{100pt}\\ \leq
      \sqrt{ 2\pi}\,\,\sup\limits_{\lambda\in J}\Big[\frac{|\sigma(\lambda)|}{|a'(\lambda)|^\frac12}
 \|A(\lambda)\|_{s,-s}^\frac12\Big]\|\phi\|_{L^2(\Rn)}.
     \endaligned\ee
 \end{itemize}
 \end{cor}

         \begin{rem}
          Note that $\|A(\lambda)\|_{s,-s}$ is the operator norm of the self-adjoint operator
          $\jp{x}^{-s}A(\lambda)\jp{x}^{-s}$ on $L^2(\Rn).$
         \end{rem}

    The statement in Theorem ~\ref{thmbasicspacetime} can be rephrased by saying that the map $\mathfrak{H}_J:E(J)\Hcal\hookrightarrow  L^2(\RR_t,\MX^\ast)$ given by
    \begin{equation}\label{equindep}
   \mathfrak{H}_J \,\,\phi= \sigma(H)e^{ita(H)}P_{ac}(H)E(J)\phi ,
    \end{equation}
is bounded, under suitable conditions.

Suppose  that $P_{ac}(H)E(J)=E(J),$ namely, that the spectrum of $H$ over the open set $J$ is purely absolutely continuous. Then ~\eqref{eqbasicspacetime} can be rewritten as

\be\label{eqspacetimeabsc}
\|\mathfrak{H}_J\|\leq \sqrt{2\pi}\,\,\sup\limits_{\lambda\in {J\setminus\Ncal}}\br{\frac{|\sigma(\lambda)|}{|a'(\lambda)|^\frac12}
 \|A(\lambda)\|_{B(\MX,\MX^\ast)}^\frac12}.
\ee

 In fact, we now show that the operator-norm of $\mathfrak{H}_J$ is given by the right-hand side of ~\eqref{eqspacetimeabsc}.

\begin{thm}\label{thmbestconst}
Suppose that  $P_{ac}(H)E(J)=E(J),$ namely, that the spectrum of $H$ over the open set $J$ is purely absolutely continuous.

Then we have
\be\label{supsupeq}\aligned
\|\mathfrak{H}_J\|=\sup_{\|\phi\|_{\Hcal}=1}\|\sigma(H)e^{it\,a(H)}P_{ac}(H)E(J)\phi\|_{ L^2(\RR_t,\MX^\ast)}
\\=\sqrt{2\pi}\,\,\sup\limits_{\lambda\in {J\setminus\Ncal}}\Big[\frac{|\sigma(\lambda)|}{a'(\lambda)^\frac12}
 \|A(\lambda)\|_{B(\MX,\MX^\ast)}^\frac12\Big].
\endaligned\ee
\end{thm}

\begin{proof} For simplicity we assume that $\Ncal=\emptyset.$ Let
  $$\Gamma=\set{\phi\in \MX\,;\quad \|\phi\|_{\Hcal}=1}.$$
  By the density of $\MX$ in $\Hcal,$ it clearly suffices to take the supremum in Equation ~\eqref{supsupeq} over $\phi\in\Gamma.$ Using Equation ~\eqref{eq2identitya} we therefore write, for $\phi\in\Gamma,$
  \be\label{suponphi}\aligned
  \Big\|\sigma(H)e^{ita(H)}P_{ac}(H)E(J)\phi\Big\|_{ L^2(\RR_t,\MX^\ast)}\hspace{100pt}\\
     =\sup\limits_{\|g\|_{ L^2(\RR_t,\MX)}=1}\Big|\int_\RR\int_{J_\lambda}\sigma(\lambda)
 e^{ita(\lambda)}\jp{A(\lambda)\phi,
 g(t)}\,d\lambda dt\Big|\\
 =\sup\limits_{\|g\|_{ L^2(\RR_t,\MX)}=1} \sqrt{2\pi}\,\Big|\int_{J_\lambda}\sigma(\lambda)
 \jp{A(\lambda)\phi,
 \widetilde{g}(a(\lambda))}\,d\lambda \Big|,
  \endaligned\ee
   where $\widetilde{g}$ is defined in ~\eqref{eqwideg}.

   Defining $G(\lambda)=a'(\lambda)^\frac12\widetilde{g}(a(\lambda)),$ it follows from  the Plancherel theorem that
   $$\|G(\lambda)\|_{ L^2(J_\lambda,\MX)}\leq\|g(t)\|_{ L^2(\RR_t,\MX)},$$
   with equality only if $\supp\,\widetilde{g}\subseteq a(J)=\set{\eta=a(\lambda)\,;\,\lambda\in J}.$

   We therefore conclude that
 \be\label{suponphi1} \aligned\Big\|\sigma(H)e^{ita(H)}P_{ac}(H)E(J)\phi\Big\|_{ L^2(\RR_t,\MX^\ast)} \\ =\sqrt{2\pi}\sup\limits_{\|G\|_{ L^2(J_\lambda,\MX)}=1}
     \Big|\int_J\frac{|\sigma(\lambda)|}{a'(\lambda)^\frac12}
 \jp{A(\lambda)\phi,
 G(\lambda)}\,d\lambda \Big|.
 \endaligned\ee

     Fix $\lambda_0\in J,$  and take $\delta>0.$ In view of ~\eqref{eqnormAXXast} there exists $\psi\in \MX$
      such that $\|\psi\|_{\MX}=1$ and
     \be\label{eqA0psidelta}\jp{A(\lambda_0)\psi,\psi}\,\,\,\geq\|A(\lambda_0)\|_{B(\MX,\MX^\ast)}-\delta.\ee

     Let  $h>0$  so that  $D_h=(\lambda_0-\frac{h}{2},\lambda_0+\frac{h}{2})\subseteq J,$ and let $\chi_{h}(\lambda)$ be the characteristic function of $D_h.$

     Take
      $$G(\lambda)=h^{-\frac12}\chi_{h}(\lambda)\psi,$$
      so that  the integral in  ~\eqref{suponphi1} becomes
      \be\label{innerintG} K_h=h^{-\frac12}\int_{\lambda_0-\frac{h}{2}}^{\lambda_0+\frac{h}{2}}\frac{|\sigma(\lambda)|}{a'(\lambda)^\frac12}
 \jp{A(\lambda)\phi,
 \psi}\,d\lambda.\ee
     In view of Remark ~\ref{remEalambda} we can take in ~\eqref{innerintG},
      $$\phi=\frac{E(D_h)\psi}{\|E(D_h)\psi\|_{\Hcal}}\,\,\,,$$
      and obtain, using the obvious fact that $\frac{d}{d\lambda}E(\lambda) E(D_h)\psi= A(\lambda)\psi,\,\,\lambda\in D_h,$
      \be\label{innerintG1} K_h=\frac{h^{-\frac12}}{\|E(D_h)\psi\|_{\Hcal}}
      \int_{\lambda_0-\frac{h}{2}}^{\lambda_0+\frac{h}{2}}\frac{|\sigma(\lambda)|}{a'(\lambda)^\frac12}
 \jp{A(\lambda)\psi,
 \psi}\,d\lambda.\ee
 By the spectral theorem
    $$\|E(D_h)\psi\|_{\Hcal}^2=\int_{\lambda_0-\frac{h}{2}}^{\lambda_0+\frac{h}{2}}\jp{A(\lambda)\psi,
 \psi}\,d\lambda.$$

  We conclude that by ~\eqref{eqA0psidelta}
  \be
  \lim\limits_{h\to 0}K_h=\frac{|\sigma(\lambda_0)|}{a'(\lambda_0)^\frac12}\,\,\jp{A(\lambda_0)\psi,\psi}^\frac12\,\,\geq \,\,\frac{|\sigma(\lambda_0)|}{a'(\lambda_0)^\frac12}\Big(\|A(\lambda_0)\|_{B(\MX,\MX^\ast)}-\delta\Big)^\frac12 .
   \ee
   From ~\eqref{suponphi1}  we now obtain
   $$\aligned\sup_{\phi\in \Gamma}\|\sigma(H)e^{it\,a(H)}P_{ac}(H)E(J)\phi\|_{L^2(\RR_t,\MX^\ast)}\\\geq
   \sqrt{2\pi}\frac{|\sigma(\lambda_0)|}{a'(\lambda_0)^\frac12}
   \Big(\|A(\lambda_0)\|_{B(\MX,\MX^\ast)}-\delta\Big)^\frac12.\endaligned$$
  Since this holds true for any $\lambda_0\in J,$
  and for any $\delta>0,$ we conclude that
  \be\aligned
  \sup_{\phi\in\Gamma}\Big\|\sigma(H)e^{ita(H)}P_{ac}(H)E(J)\phi\Big\|_{L^2(\RR_t,\MX^\ast)}\hspace{100pt}\\
  \geq \sqrt{2\pi}\sup\limits_{\lambda\in J}\Big[\frac{|\sigma(\lambda)|}{a'(\lambda)^\frac12} \|A(\lambda)\|_{B(\MX,\MX^\ast)}^\frac12\Big].
   \endaligned\ee
     The equality ~\eqref{supsupeq} now follows by combining this last estimate with the opposite one ~\eqref{eqspacetimeabsc}.
\end{proof}
\begin{rem} The result in Theorem ~\ref{thmbestconst} can be characterized as a ``best constant'' result in the spacetime estimate. In the case of the Laplacian $H=-\Delta$ such a result was obtained in ~\cite{simon}.
\end{rem}
\section{Comparison principles}\label{sec-comparison}
As an application of the spectral identities established in Section ~\ref{secglobalevolution},
we present here applications of spectral comparison
principles for self-adjoint operators. Given two such operators, these principles enable us to carry smoothing estimates for one of them to the other, provided their spectral densities can be ``compared''.  As a special case, smoothing estimates can be obtained for functions of a self-adjoint operator.

The general methodology of ``comparison principles'' was introduced in ~\cite{RS2012,RS2016}, where its usefulness was demonstrated by a wide array of operators.

We continue to work within the basic setup of Section ~\ref{secbasicsetup},
and add the hypothesis that there is another self-adjoint operator $\widetilde H$ on
the same Hilbert space $\Hcal.$
For this operator we denote by $\widetilde E(\lambda)$ its associated spectral family and by
 ${\widetilde P}_{ac}(\widetilde H)$ the projection unto its absolutely continuous spectrum.

In addition, with the same subspace $\MX\subseteq\Hcal,$ we denote by ${\widetilde A}(\lambda)=\frac{d}{d\lambda}\widetilde E(\lambda):\MX\to \MX^\ast$ the spectral density.

To simplify the statements of the results below, we assume that for an open set $J\subseteq \RR,$ the spectra of both operators are absolutely continuous in $J,$ and extend the conditions in Assumptions  ~\ref{assumption}, ~\ref{assumptionasigma}  also to $ {\widetilde H}.$

 \begin{assume}\label{assumption2} Let $J\subseteq\R$ be an open set, so
 that $P_{ac}(H)E(J)=E(J)$ and ${\widetilde P}_{ac}({\widetilde H}){\widetilde E}(J)={\widetilde E}(J).$

The operator-valued functions $A(\lambda):\MX\hookrightarrow \MX^\ast$
and ${\widetilde A}(\lambda):\MX\hookrightarrow \MX^\ast$
are weakly continuous on $J.$

The real-valued functions $\sigma(\lambda)$, $\widetilde\sigma(\lambda)$ and $a(\lambda)$, $\widetilde{a}(\lambda)$ satisfy the conditions of Assumption ~\ref{assumptionasigma}.

\end{assume}
Recall that we assume the existence of finite sets $\Ncal$ (resp. $\widetilde{\Ncal}$) where the derivatives of the amplitudes $a(\lambda)$ (resp. $\widetilde{a}(\lambda)$) may not exist. For simplicity of the presentation, we assume in the proofs below that they  are null;\,\,\,$\Ncal= \widetilde{\Ncal}=\emptyset.$
The following proposition is a straightforward consequence of
Proposition \ref{propidentity} .

\begin{prop}\label{THM:sp-comparison-loc}
Assume the conditions of Assumption ~\ref{assumption2}.

{\rm (i)} {\textbf{(space-local comparison)}}
 Let $\phi,\psi\in\MX.$
Suppose that we have
\begin{equation*}\label{EQ:sc-cond1-loc}
\frac{|\sigma(\lambda)|}{\abs{{a}'(\lambda)}^{1/2}}
\abs{
\jp{ A(\lambda)\phi,\,\psi}}
\geq
\frac{|\widetilde\sigma(\lambda)|}{\abs{\widetilde{a}'(\lambda)}^{1/2}}
\abs{
\jp{
\widetilde A(\lambda)\phi,\,\psi}
}
\end{equation*}
for almost all $\lambda\in J$.
Then we have
\begin{multline*}\label{comp2}
\n{
\p{\sigma(H)e^{ita(H)}P_{ac}(H)E(J)\phi,\,\psi}_{\Hcal}
}_{L^2(\R_t)}
\\
\geq 
\n{
\p{\widetilde\sigma(\widetilde H)e^{it\widetilde a(\widetilde H)}
{\widetilde P}_{ac}(\widetilde{H}){\widetilde E}(J)\phi,\,\psi}_{\Hcal}
}_{L^2(\R_t)}.
\end{multline*}
{\rm (ii)} {\textbf{(space-global comparison)}}
Let $\phi\in\MX$ and
suppose that
\begin{equation*}\label{EQ:sc-cond1-glo}
\frac{|\sigma(\lambda)|}{|a'(\lambda)|^{1/2}}
\n{
A(\lambda)
\phi
}_{ \MX^\ast}
\geq
\frac{|\widetilde\sigma(\lambda)|}{|\widetilde a'(\lambda)|^{1/2}}
\n{
{\widetilde A}(\lambda)
\phi
}_{ \MX^\ast}
\end{equation*}
for almost all $\lambda\in J$.
Then we have
\begin{equation*}\label{EQ:sc-comp-glo}
\n{\sigma(H)
e^{ita(H)}P_{ac}(H)E(J)\phi}_{L^2(\RR_t,\MX^\ast)}
\geq
\n{\widetilde\sigma(\widetilde H)
e^{it\widetilde a(H)}
{\widetilde P}_{ac}{\widetilde E}(J)\phi}_{L^2(\RR_t,\MX^\ast)}.
\end{equation*}
\end{prop}
\begin{rem}
   We refer to the first estimate as ``space-local comparison'' since the scalar product by $\psi$ ``localizes'' $\sigma(H)
e^{ita(H)}P_{ac}(H)E(J)\phi$ in $\Hcal .$ The absence of such a scalar product in the second estimate means that it is ``space-global comparison'' .
\end{rem}

 The following theorem is a direct consequence of Theorem ~\ref{thmbestconst} (and ~\eqref{eqnormAXXast}).

\begin{thm}[\textbf{uniform comparison}]\label{THM:sp-comparison-gen-op}
Assume the conditions of Assumption ~\ref{assumption2}.\newline
Suppose that for almost all $\lambda\in J$ and for every unit vector $\psi\in\MX\,(\|\psi\|_{\MX}=1)$ there exists a unit vector  $\widetilde\psi\in\MX\,(\|\widetilde\psi\|_{\MX}=1),$ such that
\begin{equation}\label{EQ:sc-cond1-op}
\frac{|\sigma(\lambda)|^2}{\abs{a'(\lambda)}}
\jp{A(\lambda)\psi,\,\psi}
\geq
\frac{|\widetilde\sigma(\lambda)|^2}{\abs{\widetilde a'(\lambda)}}
\jp{{\widetilde A}(\lambda)\widetilde\psi,\,\widetilde\psi}.
\end{equation}

Suppose that there exists a constant $C_0>0$ so that the following estimate holds for all $\phi\in \Hcal:$
\begin{equation}\label{EQ:cor-smest-1-op}
\n{\sigma(H)e^{ita(H)}P_{ac}(H)E(J)\phi}_{ L^2(\RR_t,\MX^\ast)}
\leq
C_0 \|\phi\|_{\Hcal}.
\end{equation}

Then it follows that
\begin{equation}\label{EQ:cor-aim2-op}
\n{\widetilde\sigma(\widetilde H)
e^{it\widetilde a(\widetilde H)}
{\widetilde P}_{ac}(\widetilde H){\widetilde E}(J)\phi}_{ L^2(\RR_t,\MX^\ast)}
\leq C_0 \|\phi\|_{\Hcal}
\end{equation}
holds for all $\phi\in \Hcal$, with the same
constant $C_0$.
\end{thm}

An important  application of Theorem \ref{THM:sp-comparison-gen-op}
is stated in the following corollary;   a smoothing estimate for some operator
$H$, yields smoothing estimates for
functions of $H$.

\begin{cor}\label{THM:powers}
Assume the conditions of Assumptions ~\ref{assumption} and ~\ref{assumptionasigma}.
Suppose also that there is a constant $C_0>0$ such that
\begin{equation}\label{EQ:cor-smest-2-op}
\n{\sigma(H)e^{itH}P_{ac}(H)E(J)\phi}_{ L^2(\RR_t,\MX^\ast)}
\leq
C_0 \|\phi\|_{\Hcal}
\end{equation}
for all $\phi\in \Hcal$.

Then we have the estimate
\begin{equation}\label{EQ:cor-aim3-op}
\n{|a'(H)|^{1/2}\sigma(H)e^{ita(H)}P_{ac}(H)E(J)\phi}_{ L^2(\RR_t,\MX^\ast)}
\leq
C_0 \|\phi\|_{\Hcal}
\end{equation}
for all $\phi\in \Hcal$, with the same
constant $C_0$.
\end{cor}

\begin{proof}
Let $H=\widetilde H$, then
the corollary  follows from
Theorem \ref{THM:sp-comparison-gen-op}, taking $a(\lambda)=\lambda$ and letting there $\widetilde a(\lambda)=a(\lambda)$ and
$\widetilde\sigma(\lambda)=\sigma(\lambda)a'(\lambda)^\frac12.$
\end{proof}
 \subsection[Comparing unperturbed and perturbed operators]{Comparing unperturbed and perturbed operators}\label{subsec-pert}
The comparison principles presented above provide a very effective way of dealing with global spacetime estimates for perturbations, as we shall now see.

Turning back to the setup in Section ~\ref{secbasicsetup}
we impose on  the self-adjoint operator  $H$   an assumption which is stronger than Assumption ~\ref{assumption}.

\begin{assume}\label{assumptionholder} Let $J\subseteq\RR$ be an open set , such that $P_{ac}(H)E(J)=E(J).$  The operator-valued function $A(\lambda):\MX\hookrightarrow \MX^\ast$ is H\"{o}lder continuous on $J,$ with respect to the operator topology of $B(\MX,\MX^\ast).$
\end{assume}
This assumption implies the Limiting Absorption Principle ~\cite[Theorem 3.6]{BA10}:
\begin{claim} Let $R(z)=(H-z)^{-1},\,\Im\,z\neq 0.$ Then the limits
 \be
    R^\pm(\lambda)=\lim\limits_{\eps\downarrow 0}R(\lambda\pm i\eps),\quad \lambda\in J,
 \ee
    exist in the uniform operator topology of $B(\MX,\MX^\ast).$
\end{claim}

As suggested in Remark ~\ref{rem-XXast}, we need to assume an abstract setup that is  more general than the one used hitherto. We follow the presentation of ~\cite[Section 4]{BA10}. The main additional ingredient is a subspace $\MX^\ast_H\subseteq \MX^\ast$ that is densely and continuously embedded.

Now let $V:\MX^\ast_H\hookrightarrow\MX$ be  a short-range and symmetric operator.  In particular, we assume the following.
\be\label{eqassumeVSR}
\begin{cases}
 V\,\,\mbox{ is compact.}\\
\mbox{The subspace}\,\, \Dcal=D(H)\cap \MX^\ast_H \,\,\mbox{ is dense in}\,\, \Hcal\,\,\\\mbox{and the restriction of}\,\, V\,\,\mbox{ to}\,\, \Dcal\,\,\mbox{ is symmetric in}\,\, \Hcal.\\
\mbox{There exist}\,\, z\in\CC,\,\Im\,z\neq 0,\,\,\mbox{ and a linear subspace}\,\, \Dcal_z\subseteq \Dcal\,\,\\\mbox{ such that the image}\,\, (H-z)\Dcal_z\,\,\mbox{ is dense in}\,\, \MX.
\end{cases}
\ee
\begin{rem}\label{rem-generaloV}
The compactness hypothesis on $V$ can be replaced by the assumption that the operators $VR^\pm(\lambda):\MX\hookrightarrow\MX$ are compact for any $\lambda\in J.$
\end{rem}
   Under these assumptions, the operator $\widetilde H=H+V$ is essentially self-adjoint on $\Dcal$ ~\cite[Lemma 4.2]{BA10} and we use $\widetilde H$ to denote its self-adjoint extension, and let $\widetilde E(\lambda)$ be its associated spectral family.
   \begin{prop}\label{propAtildelambda}
        Assume that
        the operators $\set{I+VR^\pm(\lambda)\,;\,\,\lambda\in J}$ are invertible in $\MX$ and
        \be\label{eqassumpJ2}
        \sup\limits_{\lambda\in J}\|[I+VR^\pm(\lambda)]^{-1}\|_{B(\MX,\MX)}<\infty.
        \ee

        Then the weak derivative ${\widetilde A}(\lambda)=\frac{d}{d\lambda}\widetilde E(\lambda),\,\lambda\in J,$ exists in $B(\MX,\MX^\ast).$ Furthermore, there exists a constant $C>0,$ such that
        \be
         \|{\widetilde A}(\lambda)\|_{B(\MX,\MX^\ast)}\leq C \|A(\lambda)\|_{B(\MX,\MX^\ast)},\quad \lambda\in J.
        \ee
   \end{prop}
   \begin{proof} Let ${\widetilde R}(z)=(\widetilde H -z)^{-1},\,\Im\,z\neq 0.$ Under the conditions of the Proposition, we can invoke the resolvent equation to obtain the Limiting Absorption Principle for $\widetilde H$~\cite[Eq.(4.2)]{BA10}:
   \be\label{eqAtilde}
       {\widetilde R}^\pm(\lambda)=\lim\limits_{\eps\downarrow 0}{\widetilde R}(\lambda\pm i\eps))=R^\pm(\lambda)[I+VR^\pm(\lambda)]^{-1},\quad \lambda\in J.
    \ee
    The existence of the spectral derivative follows from the well-known relation ~\cite[Corollary 3.7]{BA10}
     $${\widetilde A}(\lambda)=\frac{1}{2\pi i}[{\widetilde R}^+(\lambda)-{\widetilde R}^-(\lambda)],$$
     so by ~\eqref{eqAtilde}
     $$\aligned{\widetilde A}(\lambda)
&=A(\lambda)+\frac{1}{2\pi i}{ R}^+(\lambda)\Big[(I+VR^+(\lambda))^{-1}-I\Big]
\\ &\hspace{4cm}
-\frac{1}{2\pi i}{ R}^-(\lambda)\Big[(I+VR^-(\lambda))^{-1}-I\Big]
\\
&=A(\lambda)+A(\lambda)\Big[(I+VR^+(\lambda))^{-1}-I\Big]
\\  &\hspace{4cm}
     +\frac{1}{2\pi i}R^-(\lambda)\Big[(I+VR^+(\lambda))^{-1}-(I+VR^-(\lambda))^{-1}\Big].\endaligned$$
     Since
     $$\frac{1}{2\pi i}\Big[(I+VR^+(\lambda))^{-1}-(I+VR^-(\lambda))^{-1}\Big]=-(I+VR^+(\lambda))^{-1}VA(\lambda)(I+VR^-(\lambda))^{-1},$$
       we finally get
       \be
       {\widetilde A}(\lambda)=A(\lambda)(I+VR^+(\lambda))^{-1}-(I+VR^+(\lambda))^{-1}VA(\lambda)(I+VR^-(\lambda))^{-1},
       \ee
       and the proof is complete in view of the assumption ~\eqref{eqassumpJ2}.
   \end{proof}
   \begin{rem} Note that in Proposition ~\ref{propAtildelambda} we did not need any uniform (in $\lambda\in J$) estimates on the limiting resolvent values $R^\pm(\lambda).$

   The invertibility assumption of $(I+VR^\pm(\lambda))^{-1}$ already follows from the assumption that ${\widetilde H}$ has no eigenvalues in $J$  and an additional mild assumption on $A(\lambda)$ ~\cite[Theorem 4.13]{BA10}.
   \end{rem}
   In view of Proposition ~\ref{propAtildelambda} we can formulate a version of Theorem ~\ref{THM:sp-comparison-gen-op} for perturbations:
   \begin{thm}[\textbf{Spacetime estimates for perturbed operators}]\label{THM:sp-comparison-gen-op-pert} Let $H$ be a self-adjoint operator on $\Hcal$ satisfying
 Assumption ~\ref{assumptionholder}.\newline
 Let $\widetilde H=H+V$ where $V$ is a symmetric short-range perturbation so that the hypotheses of Proposition ~\ref{propAtildelambda} are satisfied.

Assume further that the real-valued functions $\sigma(\lambda)$, $\widetilde\sigma(\lambda)$ and $a(\lambda)$, $\widetilde{a}(\lambda)$ satisfy the conditions of Assumption ~\ref{assumptionasigma} and
\begin{equation}\label{EQ:sc-cond1-op-pert}
\frac{|\sigma(\lambda)|^2}{\abs{a'(\lambda)}}
\geq
\frac{|\widetilde\sigma(\lambda)|^2}{\abs{\widetilde a'(\lambda)}},\quad \lambda\in J.
\end{equation}

Suppose that there exists a constant $C_0>0$ so that  for all $\phi\in \Hcal$ we have
\begin{equation}\label{EQ:cor-smest-1-op-pert}
\n{\sigma(H)e^{ita(H)}E(J)\phi}_{ L^2(\RR_t,\MX^\ast)}
\leq
C_0 \|\phi\|_{\Hcal}.
\end{equation}

Then there exists a constant $C>0$ such that
\begin{equation}\label{EQ:cor-aim2-op-pert}
\n{\widetilde\sigma(\widetilde H)
e^{it\widetilde a(\widetilde H)}
{\widetilde E}(J)\phi}_{ L^2(\RR_t,\MX^\ast)}
\leq CC_0 \|\phi\|_{\Hcal}
\end{equation}
holds for all $\phi\in \Hcal.$
\end{thm}

\section{Applications to operators of mathematical physics}\label{sec-examples}
We shall now give a few examples that illustrate the scope of the abstract results when applied to a variety of  operators that are frequently studied in mathematical physics.
\subsection{The fractional Laplacian}
 Consider the operator $H=-\Delta$ in $\Hcal=L^2(\Rn)$, $n\geq 3.$ It is absolutely continuous in $J=(0,\infty)$ and the condition of Assumption ~\ref{assumptionholder} is satisfied with $\MX=L^2_s(\Rn),\,s>\frac12$ ~\cite[Section 2]{ba9}.

 Furthermore, it was shown in ~\cite[Theorem 1]{ba5} that for all $\phi\in L^2(\Rn),$
 $$\int_{\RR}\int_{\Rn}\jp{x}^{-2s}|H^{\frac14}e^{itH}\phi(x)|^2\,dx\,dt\leq C_0\|\phi\|_{L^2(\Rn)}^2.$$
    Consider now the operator $H^\alpha$ for some $\alpha>0.$ Then it follows from Corollary ~\ref{THM:powers} that
    \be\label{eqfraclaplacian}
   \alpha^2 \int_0^\infty\int_{\Rn}\jp{x}^{-2s}|H^{\frac{2\alpha-1}{4}}e^{itH^\alpha}\phi(x)|^2\,dx\,dt\leq C_0\|\phi\|_{L^2(\Rn)}^2,
    \ee
    for all $\phi\in L^2(\Rn).$
\subsection{The Stark Hamiltonian}
Consider the self-adjoint operator in $\Hcal=L^2(\Rn)$, $n\geq 1,$
$$H=-\Delta-x_1,$$
  where $x=(x_1,x_2,\ldots,x_n)\in \Rn.$

  This operator is the well-known ``Stark Hamiltonian'', describing the motion of a quantum-mechanical charged particle in a uniform electric field (all physical constants scaled to unity).

  It follows from ~\cite[Lemma 2.2]{ba10} that $H$ satisfies the condition of Assumption ~\ref{assumption} with
  $$J=\RR,\quad \MX=L^2_{\frac34}(\Rn). $$
  In fact, using the notation introduced at the end of Section ~\ref{secbasicsetup} it holds that
  $$ \|A(\lambda)\|_{\frac34,-\frac34}\leq C(1+|\lambda|)^{-\frac12},\quad \lambda\in\RR.$$
  It follows that ~\cite[Theorem A]{ba10}
  $$\int_{\RR}\int_{\Rn}\jp{x}^{-\frac32}|(I+|H|)^\frac14 e^{itH}\phi(x)|^2\, dxdt\leq C_0\|\phi\|_{L^2(\Rn)}^2. $$
   For the operator $|H|^\alpha,$ with $\alpha>0,$ Corollary ~\ref{THM:powers} now yields
   \be\label{eqstarkalpha}\alpha^2\int_{\RR}\int_{\Rn}\jp{x}^{-\frac32}||H|^{\alpha-1}(I+|H|)^\frac14 e^{it|H|^\alpha}\phi(x)|^2\, dxdt\leq C_0\|\phi\|_{L^2(\Rn)}^2.
      \ee
      Of course, such an estimate can be derived for any function $a(H)$ where $a(\lambda)$ satisfies the condition of Assumption ~\ref{assumptionasigma}.
\subsection{The Schr\"{o}dinger operator with potential}
Consider the operator $\widetilde H=H+V$ in $L^2(\Rn),\,n\geq 3,$  where $H=-\Delta.$  This operator can be studied in terms of Theorem ~\ref{THM:sp-comparison-gen-op-pert}.

Employing the notation of Subsection ~\ref{subsec-pert} we let $\MX=L^2_{s}(\Rn),$ so that $\MX^\ast=L^2_{-s}(\Rn)$ and $\MX^\ast_H=H^2_{-s}(\Rn),$
 the Sobolev space of functions whose derivatives up to second-order are in $L^2_{-s}(\Rn).$

 The condition of Assumption ~\ref{assumptionholder} is satisfied with $s>\frac12$ ~\cite[Section 2]{ba9}.

We assume that $V(x)$ is a  symmetric, short-range potential (see ~\eqref{eqassumeVSR}).

 The limiting values  of the resolvent of the free operator $H$ satisfy the classical Agmon property:
 \begin{claim}
   For any $s>\frac12,$ the limiting values $R^\pm(\lambda)=\lim\limits_{\eps\downarrow 0}(H-\lambda\mp i\eps)^{-1},\,\lambda>0$ satisfy
   \be\label{eqestRpmmss}\|R^\pm(\lambda)\|_{s,-s}\leq C\lambda^{-\frac12},\quad \lambda>0,\ee
   so that the spectral derivative $A(\lambda)=\frac{1}{2\pi i}(R^+(\lambda)-R^-(\lambda))$ also satisfies
   \be\label{eqestAlamss}\|A(\lambda)\|_{s,-s}\leq C\lambda^{-\frac12},\quad \lambda>0,\ee
   where $C>0$ depends only on $s,n.$
 \end{claim}

 We  now obtain the following spacetime estimate for $\widetilde H.$
 \begin{prop} Let $s>\frac12$ and assume that the potential $V$ is short-range and symmetric, and
       the spectrum of $\widetilde H$ in $(0,\infty)$ is absolutely continuous.

       Assume further that
       \be\label{eqlimsupVR} \limsup\limits_{\lambda\to\infty}\|VR^\pm(\lambda)\|_{s,s}<1.
       \ee

        Fix $\delta>0$ and let $J_\delta=(\delta,\infty).$ Then there exists a constant $ C_{\delta,s}$ so that
         \be\label{eqSchrVx}
      \int_{\RR}\int_{\Rn}\jp{x}^{-2s}|(I+\widetilde H)^\frac14 {\widetilde  E}(J_\delta)e^{it{\widetilde H}}\phi(x)|^2\,dx\,dt\leq C_{\delta,s}\|\phi\|_{L^2(\Rn)}^2,
      \ee
      for all $\phi\in L^2(\Rn).$

      \end{prop}
      \begin{proof}
        The estimate ~\eqref{eqSchrVx}, with the free Schr\"{o}dinger operator $H$ instead of $\widetilde H$ (and $E(\lambda)$ instead of ${\widetilde E}(\lambda)$) is well-known ~\cite[Theorem A]{ba9}.

        In order to invoke Theorem ~\ref{THM:sp-comparison-gen-op-pert} we need to verify the condition in Proposition ~\ref{propAtildelambda},
        namely the estimate ~\eqref{eqassumpJ2}.

        Since we assume that $\widetilde H$ has no eigenvalues in $(0,\infty)$ the operators $I+VR^\pm(\lambda)$ are invertible (in $\MX$) for any $\lambda>0,$ and being continuous (in the operator topology) we conclude that, for any $\Lambda>\delta$ we have
        \be\label{eqsupVRlambdaLambda}\sup\limits_{\lambda\in[\delta,\Lambda]}\|[I+VR^\pm(\lambda)]^{-1}\|_{s,-s}<\infty.\ee
        On the other hand, it follows from the assumption ~\eqref{eqlimsupVR} that if $\Lambda$ is sufficiently large there exists $\eta>0$ so that
            $\|VR^\pm(\lambda)\|_{s,s}<1-\eta,$ so that
            \be\label{eqsupVRlLambdainf}\sup\limits_{\lambda>\Lambda}\|[I+VR^\pm(\lambda)]^{-1}\|_{s,-s}<\frac{1}{\eta}.\ee
            Combining ~\eqref{eqsupVRlambdaLambda} and ~\eqref{eqsupVRlLambdainf} we get the needed estimate ~\eqref{eqassumpJ2}.
      \end{proof}
      \begin{example}\begin{itemize} \item Consider the case where $V=V(x)$ is a multiplication operator. In view of ~\eqref{eqestRpmmss} the condition ~\eqref{eqlimsupVR} is certainly satisfied under the pointwise decay condition

 \be\label{eqVshort}|V(x)|\leq C\jp{x}^{-1-\eps},\quad x\in \Rn,\,\,\eps>0.\ee
    \item More generally, let $H^\alpha_{-s}(\Rn)$ be the weighted Sobolev space of order $\alpha>0,$ normed by
    $$\|f\|_{\alpha,s}=\|(I-\Delta)^{\frac{\alpha}{2}}f\|_{0,s}.$$
    It was shown in  ~\cite{ba6} that the condition ~\eqref{eqlimsupVR} is satisfied if for some $\eps>0$ the operator
      \be\label{eqVH1-eps}  V:H^{1-\eps}_{-s}\hookrightarrow L^2_s,\quad s>\frac12,
      \ee
      is bounded.
      \item Observe that $V$ in ~\eqref{eqVH1-eps} need not be a multiplicative real potential. For example, we can take $V$ to be the pseudodifferential operator
      \be
      V=\jp{x}^{-s}(I-\Delta)^\beta \jp{x}^{-s},
      \ee
      for any $\beta<\frac12.$
      \end{itemize}
   \end{example}
      We refer the reader to ~\cite{JK79,J80,ky} for related spacetime estimates for the Schr\"{o}dinger operator, using very different methods.

%

%

\end{document}